\documentclass{amsart}
\newtheorem{thm}{Theorem}[section]

\newtheorem{lemma}[thm]{Lemma}

\begin{document}


\title{Cubic Thue Equations}

\author{Shabnam Akhtari}
\address{Shabnam Akhtari\newline
Max-Planck-Institut f\"{u}r Mathematik\newline
 P.O.Box: 7280,
53072 Bonn,
Germany\newline}
 \email {akhtari@mpim-bonn.mpg.de}

\begin{abstract}
 We revisit a work by R.~Okazaki and prove that for every cubic binary form $F(x, y)$ with large enough discriminant, the Thue equation $|F(x , y)| = 1$ has at most $7$ solutions in integers $x$ and $y$. 
\end{abstract}

\maketitle

\section{Introduction}\label{6s1}

Let $F(x , y)$ be an irreducible binary cubic form with integral coefficients and negative discriminant. More than 80 years ago, Delone and Nagell established independently  that the equation
\begin{equation}\label{61.2}
|F(x , y)| = 1
\end{equation}
has at most five solutions in integers $x$, $y$. This result is proved by considering units in the algebraic number field $\mathbb{Q} (\rho)$, where $\rho$ is the real root of $F(x , 1) = 0$.  In their proofs  the fact that the group of units in the ring of integers of $\mathbb{Q}(\rho)$ is generated by one fundamental unit is essential.

The situation where the discriminant of $F(x , y)$ is positive is complicated by the fact that the number field $\mathbb{Q}(\rho)$ (where $\rho$ is any real root of $F(x , 1) = 0$) has a ring of integers generated by a pair of fundamental units. However, it is possible to reduce (\ref{61.2}) to a set of exponential equations to which a local method of Skolem can be applied. In this way, Ljunggren  \cite{6Lj} and Baulin \cite{6Ba}, solved (\ref{61.2})  for $F(x , y) = x^{3} - 3 x y^{2} + y^{3}$ of discriminant $81$ and $F(x , y) = x^{3} + x^{2}y -2xy^{2} - y^{3}$ of discriminant  $49$, respectively. In the first case there are $6$ solutions and in the second case there are $9$ solutions to (\ref{61.2}).

In 1929, Siegel \cite{6Si} used the theory of Pad\'e
approximation to binomial functions (via the hypergeometric functions), to show for $F$ cubic of positive discriminant, that equation (\ref{61.2}) has at most $18$ solutions in integers $x$ and $y$. Refining these techniques, Evertse \cite{6Ev} reduced this upper bound to $12$. Later, Bennett \cite{6Ben} showed that if $F(x, 1)$ has at least two distinct complex roots, then the equation $F(x , y) = 1$ possesses at most $10$ solutions in integers $x$ and $y$. In 2003, by studying the geometry of numbers in the ``logarithmic space'', Okazaki \cite{6Ok} proved that if discriminant of $F$ is greater than $5.65 \times 10^{65}$ then equation (\ref{61.2}) has at most $7$ solutions. Okazaki's method is essentially different from Evertse's. 
In this paper, we will  relate some geometric ideas of Okazaki \cite{6Ok} to the method of Thue-Siegel as refined by Evertse \cite{6Ev}, in conjunction with lower bounds for linear forms in logarithms of algebraic numbers. The following are the main results of this paper:
\begin{thm}\label{6main} 
  If $F(x , y)$ is a binary cubic form with discriminant $D > 1.4 \times 10^{57}$,  then the equation
  $$
\left|F(x , y)\right| = 1
$$
  possesses at most $7$ solutions in integers $x$ and $y$.
  \end{thm}
  
\begin{thm}\label{66}
Let $F(x , y)$ be a cubic form with discriminant $D > 9 \times 10^{58}$. If $F (x , y)$ is equivalent to a reduced form which is not monic, then the equation $F(x , y) = 1$ posses at most $6$ solutions in integers $x $ and $y$.
\end{thm}

Despite the numerical improvement, the bounds remain out of reach of computers. The main purpose of this paper, is to look at some beautiful geometric ideas of Okazaki  \cite{6Ok} from Classical Analysis point of view.

In 1990, using the fact that the underlying number fields are the so-called ``simplest cubics", Thomas \cite{6Tho} showed that  the equations 
$$
G_{1 , n}(x , y) = x^{3} + n x^{2} y - (n + 3) x y^{2} + y ^{3}  = 1
$$
 have only the solutions $(1, 0)$ , $(0 , 1)$ and $(-1 , -1)$ in integers, provided $n \geq 1.365 \times 10^{7}$.  This restriction was later removed by Mignotte \cite{6Mig} except for the equation with 
$$
n \in \{ -1, 0 ,2\}.
$$
 It is known that $G_{1 , n}(x , y) = 1$ has $9$ solutions for $n = -1$  (\cite{6Ba}), $6$ solutions for $n = 0$ (\cite{6Lj}) and $6$ solutions for $n =1$  (\cite{6Ga}) .

Define $F_{m}(x , y)$ by 
$$
F_{m}(x , y) = x^{3} - (m + 1) x^{2} y + m x y^{2} + y^{3}
$$
for $m \in \mathbb{Z}$. Provided $m \neq -2$  , $-1$ or $1$ , the equation $F_{m}(x , y) = 1$ has the five distinct integral solutions $(x , y) = (1 , 0)$ , $(1 , 1)$ , $(1 , -m -1)$ , $(0 , 1)$  and $(m , 1)$. That this list is complete was proven, independently, by Lee \cite{6Lee} and Mignotte and Tzanakis \cite{6Tz}, for $m$ suitably large and later, by Mignotte \cite{6Mig1}, for $m > 2$. The cases $m = 0$ and $m =1$ correspond to discriminant $-23$ and $-31$, respectively.   

All known irreducible cubic forms $F(x , y)$, for which the equation (\ref{61.2}) has more than $5$ solutions, have discriminant less than $362$.

The following conjecture is essentially due to Nagell and refined by Peth\"o and Lippok.

\textbf{Conjecture} If $F$ is a binary cubic form with positive discriminant $D_{F}$, then the number of solutions of equation (\ref{61.2}) is less than $6$, if $D_{F}  >  361$.


\section{The Covariants  of Binary Cubic Forms }\label{6s2}
Suppose  
$$F = ax^{3} + bx^{2}y + cxy^{2} + dy^{3}$$ 
with discriminant  
$$D_{F} = 18 a b c d + b^{2} c^{2} - 27 a^{2} d^{2} - 4 a c^{3} - 4 b^{3} d = a^{4} \prod_{i , j} (\alpha_{i} - \alpha_{j})^{2} 
$$
where $\alpha_{1}$, $\alpha_{2}$ and $\alpha_{3}$ are the roots of polynomial $F(x , 1)$.
Let us define, for the form $F$, an associated quadratic form, the Hessian $H = H_{F}$, and a cubic form $G = G_{F}$, by 
\begin{eqnarray*}
H(x , y) & = & - \frac{1}{4} \left(\frac{\delta^{2}F}{\delta x^{2}}\frac{\delta^{2}F}{\delta y^{2}} -\left( \frac{\delta ^{2}F}{\delta x \delta y}\right)^{2}\right) \\ \nonumber
& = &  A x^{2} + B x y + C y^{2} 
\end{eqnarray*}
and
$$
G(x,y) = \frac{\delta F}{\delta x} \frac{\delta H}{\delta y} - \frac{\delta F}{\delta y} \frac{\delta H}{\delta x}.
$$
These forms satisfy a covariance property; i.e.
$$
H_{F \circ \gamma } = H_{F} \circ \gamma \qquad  \textrm{and} \qquad  G_{F \circ \gamma } = G_{F} \circ \gamma
$$
for all $\gamma \in GL_{2}(\mathbb{Z})$.

We call forms $F_{1}$ and $F_{2}$ equivalent if they are equivalent under $GL_{2} (\mathbb{Z})$-action; i.e. if there exist integers $a_{1}$ , $a_{2}$ , $a_{3}$ and $a_{4}$ such that 
$$
F_{1}(a_{1} x + a_{2}y , a_{3}x + a_{4}y) = F_{2} (x , y) 
$$
for all $x$, $y$ , where $a_{1} a_{4} - a_{2} a_{3} = \pm 1$.

We denote by $N_{F}$ the number of solutions in integers $x$ and $y$ of the Diophantine equation (\ref{61.2}). If $F_{1}$ and $F_{2}$ are equivalent, then $N_{F_{1}} = N_{F_{2}}$  and  $D_{F_{1}} = D_{F_{2}}$. Therefore, we can  assume that $F$ is monic (the coefficient of $x^{3}$ in $F(x , y)$ is 1).

For
$F(x,y) = ax^{3} + bx^{2}y + cxy^{2} + dy^{3}$ with discriminant $D$, it follows by routine calculation that 
$$A = b^{2} -3 a c  ,\ B = b c - 9a d ,\   C = c^{2} -3 b d $$
and 
$$ B^{2} - 4 AC = -3 D .$$

Further, these forms are related to $F (x , y)$ via the identity
\begin{equation}\label{65.1}
4 H(x , y)^{3} = G (x , y) ^{2} + 27 D F(x , y) ^{2}
\end{equation}

Binary cubic form $F$ is called \emph{reduced}  if the Hessian  
$$H(x , y) = A x^{2} + B x y + C y^{2}$$ 
of $F$,
satisfies
$$
C \geq A \geq |B|
$$
It is a basic fact (see \cite{6Del}) that every cubic form of positive discriminant is equivalent to a reduced form  $F(x , y)$ .
The reader is directed to \cite{6Del} (chapter III and supplement I) for more details on reduction.
We will later use the following lemma  to bound the discriminant $D$ from above.
\begin{lemma}\label{6mike}
Let $F$ be an irreducible binary cubic form with positive discriminant $D$ and Hessian $H$.  For all integer solutions $(x_{1} , y_{1})$ to  equation $F(x , y) = 1$, except possibly one solution, we have $H(x_{1} , y_{1}) \geq \frac{1}{2}\sqrt{3D} $.
\end{lemma} 
\begin{proof}
If $F_{1}$ is an equivalent reduced form  to $F$ and 
$$F_{1}(a_{1} x + a_{2}y , a_{3}x + a_{4}y) = F (x , y),$$
 then 
$$H_{1}(a_{1} x + a_{2}y , a_{3}x + a_{4}y) = H (x , y),$$
 where $H$ and $H_{1}$ are the Hessians of $F$ and $F_{1}$ respectively. This means the set of  values of the Hessian at solutions is fixed under $GL_{2}(\mathbb{Z})$-action. So we may assume that $F$ is reduced.
Now following the proof of lemma 5.1. of \cite{6Ben}, we suppose $(x , y)$ is a solution to $F(x , y) = 1$ with $y \neq 0$.  If $|y| \leq |x| $, then, since $A > |B| $  and $B^{2} - 4AC = -3D$, we have that
$$
H(x , y) = A x^{2} + B x y + C y^{2} \geq Cy^{2} \geq C \geq \frac{1}{2}\sqrt{3D} . 
$$
If, on the other hand, $|y| \geq |x| + 1$,  then 
$$
H(x , y) \geq (C - |B|)y^{2} + |B| |y| + A x^{2}.
$$
Since this is an increasing function of $|y|$ and $y \neq 0$, we have
$$
H(x , y) \geq C + A x^{2} \geq C \geq \frac{1}{2} \sqrt{3D} .
$$ 
Therefore,  if $H(x , y) < \frac{1}{2} \sqrt{3D}$, then $y = 0$ and so $x = \pm 1$ accordingly .
\end{proof}

\textbf{Remark.} The above proof shows  the only possibility for the Hessian $H(x , y)$ to assume a value less than $\frac{1}{2}\sqrt{3D}$, at a pair of solutions $(x , y)$, is when the equivalent reduced form is monic. This is because $(1 , 0)$ is a solution to (\ref{61.2}) if and only if $F$ is monic.


\section{Some Functions In The Number Field   $\mathbb{Q}(\sqrt{-3D})$}\label{6s3}

Let $\sqrt{-3D}$ be a fixed choice of the square-root of $-3D$. we will work in the number field $M = \mathbb{Q}(\sqrt{-3D})$. 
It is well-known that if $F$ has positive discriminant  then $H$ is positive definite.
By (\ref{65.1}), we may write 
$$H (x , y)^3 =  U(x, y ) V(x, y)$$
 where 
$$ 
U(x , y) = \frac{G(x , y) + 3\sqrt{-3D} F(x , y)}{2} ,
 $$
$$
 V(x , y) = \frac{G(x , y) - 3\sqrt{-3D} F(x , y)}{2}.
 $$ 
Then $ U$ and $V$ are cubic forms with coefficients belonging to $M$ such that corresponding coefficients of $U$ and $V$ are complex conjugates. Since $F$ must be also irreducible over $M$,  $U$ and $V$ do not have factors in common. It follows that $U(x , y)$ and $V(x , y)$ are cubes of linear forms over $M$, say $\xi (x , y)$ and $\eta (x,y )$.

Note that $\xi (x , y) \eta (x , y) $ must be a quadratic form which is cube root of $H(x , y) ^{3}$ and for which the coefficient of $x^{3}$ is a positive real number. Hence we have 
$$
\xi (x , y) ^{3} - \eta (x , y)^{3}   =  3\sqrt{-3D} F(x , y),
$$
\begin{equation}\label{63}
\xi(x , y)^{3} + \eta (x , y)^{3}  = G (x, y),
\end{equation}
$$
 \xi(x , y) \eta (x , y) = H (x , y).
 $$
and 
$$
\frac{\xi (x , y)} {\xi (1 , 0)}  \qquad  \textrm{and} \qquad   \frac{\eta (x , y)}{\eta(1, 0)} \ \in M .
$$
The reason for the last identity is that for any pair of rational integers $x_{0}$ , $y_{0}$, 
$$\xi (x_{0} , y_{0})\,   \qquad  \textrm{and} \qquad  \eta (x_{0} , y_{0})
$$
 are complex conjugates and the discriminant  of $H$ is $ -3D$.

We call a pair of forms $\xi$ and $\eta$ satisfying the above properties a pair of \emph{resolvent} forms.  Note that there are exactly three pairs of resolvent forms, given by 
$$(\xi , \eta), \,  (\omega \xi , \omega ^{2} \eta), \, (\omega^{2}\xi , \omega \xi),
$$
 where $\omega$ is a primitive cube root of unity.

We say that a pair of rational integers $(x , y)$ is related to a pair of resolvent forms if 
\begin{equation}\label{6related}
\left|1 - \frac{\eta(x , y)}{\xi(x , y)}\right| = \min_{0 \leq k \leq 2} \left|\omega ^{k} -\frac{\eta(x , y)}{\xi(x , y)}\right| 
\end{equation}

Following a discussion of Delone and Faddeev in \cite{6Del}, we call the roots $\rho _{1}, \rho' _{1}, \rho''_{1}$ of the equation $F(x , a) = 0$ the left roots of the form $F$, while the roots $\rho_{2}, \rho '_{2}, \rho ''_{2}$ of $F(d , -y)$ are called the right roots of the form $F$. If $t_{1}$ is a left root, then it is easily seen that $t_{2} = - ad / t_{1}$ is a right root of $F$. Two such roots of $F$ will be called corresponding roots and we will assume that $\rho_{1}$ and $\rho_{2}$ , $\rho'_{1}$ and $\rho'_{2}$, $\rho''_{1}$ and $\rho''_{2}$ correspond in pairs.

The following lemma is a statement of Lagrange's method for solution of cubic equations by means of the resolvent  adapted to the case of binary cubic forms.
\begin{lemma}\label{6lag}
For the cubic form $F(x , y)$ the following identity holds
$$
F(x , y) = \frac{1}{3\sqrt{-3D}} (\xi^{3} -\eta^{3}), 
$$
where 
$$
\xi = \xi _{1} x + \xi_{2} y,$$
 $$\eta = \eta_{1} x + \eta _{2} y,$$
 $$
\xi_{1} = \rho_{1} + \omega \rho'_{1} + \omega ^{2} \rho ''_{1},
$$
$$\eta_{1} = \rho_{1} + \omega ^{2}\rho'_{1} + \omega  \rho ''_{1},$$
  $$
\xi_{2} = \rho_{2} + \omega \rho'_{2} + \omega ^{2} \rho ''_{2},$$
  $$\eta_{2} = \rho_{2} + \omega ^{2}\rho'_{2} + \omega  \rho ''_{2}$$
 and $\omega = e^{\frac{2\pi i}{3}}$.
\end{lemma}
\begin{proof}
One can find the complete proof of Lemma \ref{6lag} in \cite{6Del}.
\end{proof}

We continue with the following definitions of $p$, $q$  and  $u_{i}$ :
$$
p= \frac{\eta + \xi}{\sqrt{2}} \  , \   q= \frac{\sqrt{-1}(\eta - \xi)}{\sqrt{2}},
$$ 
\begin{equation}\label{6du}
 u_{1} = D^{-1/6}( \frac{q}{\sqrt{6}} + \frac{p}{\sqrt{2}})  \   , \  
u_{2} =  D^{-1/6}( \frac{q}{\sqrt{6}}  -  \frac{p}{\sqrt{2}} )\   , \ 
u_{3} =  D^{-1/6}\frac{2}{\sqrt{6}}q.
\end{equation}
Since $\eta$ and $\xi$ are linear functions of $x$ and $y$, so are $p$, $q$ and $u_{i}$.
The reason for our interest in the new functions $p(x , y)$, $q(x , y)$ and $u_{i}(x , y)$, despite their apparent complication, is that they explain the relation between the method of Evertse \cite{6Ev} and the method of Okazaki \cite{6Ok}  for finding an upper bound for the number of integral solutions of (\ref{61.2}). In other words, these functions allow us to recast the resolvent forms $\xi$ and $\eta$ in a  geometric setting.

 By Lemma \ref{6lag}, we have 
$$
 \frac{q}{\sqrt{6}} = \frac { \sqrt{-1}(\eta - \xi)}{2\sqrt{3} }= \sqrt{-1}\, \frac{\omega^{2} - \omega}{2\sqrt{3}}[(\rho'_{1} - \rho''_{1})x - (\rho'_{2} - \rho''_{2})y].
 $$
We also have
$$
\omega^{2} - \omega = \cos(4\pi/3) + \sqrt{-1} \sin(4\pi/3) - (\cos(2\pi/3) + \sqrt{-1}\sin(2\pi/3)) = \sqrt{-3}.
$$
 so we get
\begin{equation}\label{6q}
 \frac{q}{\sqrt{6}} =- \frac{(\rho'_{1} - \rho''_{1})x  +  (\rho'_{2} - \rho''_{2})y}{2}.
 \end{equation}
 Further
 \begin{eqnarray*}
 & & \frac{p}{\sqrt{2}} = \\ \nonumber
& & \frac{\left(2\rho_{1} + \omega (\rho'_{1} + \rho''_{1} ) +\omega^{2} ( \rho'_{1} +\rho''_{1}) \right) x + \left(2\rho_{2} + \omega (\rho'_{2} + \rho''_{2} ) +\omega^{2} ( \rho'_{2} +\rho''_{2}) \right) y} {2}.
 \end{eqnarray*}
  Since $\omega$ is a primitive third root of unity, $\omega + \omega^{2} = -1$. Hence
  \begin{equation}\label{6p}
 \frac{p}{\sqrt{2}} = \frac{2(\rho_{1} x + \rho _{2} y) -(\rho'_{1} + \rho''_{1})x -(\rho'_{2} + \rho''_{2})y}{2}
 \end{equation}
 Substituting  $-ad/\rho_{1}$, $-ad/\rho'_{1}$ and $-ad/\rho''_{1}$ for  $\rho_{2}$,  $\rho'_{2}$ and $\rho''_{2}$ respectively, and noting that $\rho_{1} \rho'_{1} \rho''_{1} = -a^{2}d$, we obtain the following identities:
\begin{equation}\label{6u1}
u_{1} = D^{-1/6}(\rho_{1} - \rho''_{1}) (x - \rho'_{1} y/a) ,
\end{equation}
\begin{equation}\label{6u2}
u_{2} = D^{-1/6}(\rho'_{1} - \rho_{1}) (x - \rho''_{1} y/a) ,
\end{equation}
\begin{equation}\label{6u3}
u_{3} = D^{-1/6} (\rho'_{1} - \rho''_{1}) (x - \rho_{1} y/a) ,
\end{equation}
where $\rho_{1}$, $\rho'_{1}$  and  $\rho''_{1}$ are left roots of $F$. Here we note that if we start with another choice of resolvent forms, only the order of $u_{i}$ changes. In other words, all three resolvent forms can be indexed so that
$$
 q_{i}= \frac{\sqrt{-1}(\eta_{i} - \xi_{i})}{\sqrt{2}}
$$
   Let us assume that $F$ is monic, as we may. Therefore
  $$
  (x - \rho_{1} y) (x - \rho'_{1} y) (x - \rho''_{1} y)  =  F( x ,y ) .
  $$ 
     If the pair $(x_{0} , y_{0})$ is a solution to (\ref{61.2}), we conclude that  $(x_{0} - \rho_{1} y_{0})$,
   $(x_{0} - \rho'_{1} y_{0})$ and  $ (x_{0} - \rho''_{1} y_{0})$   are units in $\mathbb{Q}(\rho_{1})$.
    Moreover,  
   $$
   u_{1}u_{2}u_{3} = D^{-1/2}(\rho_{1} - \rho''_{1})(\rho'_{1} - \rho_{1}) (\rho'_{1} - \rho''_{1}) F(x , y) = 
   \pm F(x,y).
    $$
  Suppose that  $(x , y)$ is a solution to $F(x , y) = 1$. Since
    $$
     \log|u_{1}| - \log |u_{2}| = \log\left|\frac{\rho_{1} - \rho''_{1}}{\rho'_{1} - \rho_{1}}\right| + 
   \log \left|\frac{x - \rho'_{1}y} {x - \rho''_{1}y}\right| 
   $$
   and 
$|\frac{x - \rho'_{1}y} {x - \rho''_{1}y}|$ is a unit, we can write
\begin{equation}\label{6fu}
\log|u_{1}| - \log |u_{2}| = \log\lambda_{1} + m \log\lambda_{2} + n\log \lambda _{3},
\end{equation}   
where
$\lambda_{1} =|\frac{\rho_{1} - \rho''_{1}}{\rho'_{1} - \rho_{1}}|$, $\lambda_{2}$ and $\lambda_{3}$ are  fundamental units in the ring of integers of $\mathbb{Q}(\rho_{1})$ (when $D_{F} > 0$, the number field $\mathbb{Q}(\rho_{1})$ is real and has a ring of integer generated by a pair of fundamental units). 
   
Let us fix a resolvent forms $(\xi_{i} , \eta_{i})$  and corresponding $p_{i}$ and $q_{i}$. We get 
$$
\left|1 - \frac{\eta_{i}}{\xi_{i}}\right| = \left|1 - \frac{p_{i} -\sqrt{-1} q_{i}}{p_{i} +\sqrt{-1}q_{i}}\right| = \frac{2|q_{i}|}{ |\xi_{i}|}.
$$
By identities in (\ref{63}) and Lemma \ref{6lag},  $ |\eta_{i} | = |\xi_{i}|$ and  $|\xi_{i}(x , y)| = \sqrt{H(x , y)}$. Hence,
$$
\left|1 - \frac{\eta_{i}}{\xi_{i}}\right| = \frac{2 |q_{i}|}{\sqrt{H}} .
 $$

Suppose that $(x,y)$ is a solution to (\ref{61.2})  and related to resolvent form $(\xi_{i} , \eta_{i})$. Since  
$$|1 - \frac{\eta_{i}}{\xi_{i}}|  = \min_{k=1,2,3} \left|1 - \frac{\eta_{k}}{\xi_{k}}\right|,
$$ we conclude that  
\begin{displaymath}
|q_{i}| =  \min_{k=1,2,3} |q_{k}|.
\end{displaymath}
On the other hand,  
\begin{eqnarray*}
\prod_{k=1} ^{3} |q_{i}| &=& \frac{ |\eta - \xi| |\omega \eta - \omega^{2}\xi| |\omega^{2}\eta - \omega \xi|}{2\sqrt{2}}\\
& =&\frac{|\eta^{3} - \xi^{3}|}{ 2\sqrt{2}}= \frac{3\sqrt{3}}{2\sqrt{2}} \sqrt{D} ,
\end{eqnarray*} 
where the last equality comes from the equation (\ref{63}).

If the solution $(x,y)$ is related to $(\xi_{i}, \eta_{i})$, then
\begin{equation}\label{q}
|u_{3}| = |D^{-1/6}\frac{2}{\sqrt{6}}q_{i}| <1. 
\end{equation}
 So we have 
$$\log |u_{3}|  < 0 . $$
The identity 
$$|u_{1}(x,y) u_{2}(x , y) u_{3}(x , y)| = 1$$
  holds when $(x , y)$ is a pair of solution to $|F(x , y)| = 1$. Therefore,
$$
\log |u_{1}| +\log |u_{2}| + \log |u_{3}| = 0
$$
and 
$$
\log|u_{1} u_{2}| > 0.
$$

 \section{Geometric Gap Principles}\label{6s4}
  
  We will study the geometric properties of the functions $u_{i}$ defined in section 2, by considering the well-known  geometric properties of the unit group $U$ of $\mathbb{Q}(\rho_{1})$, where $\rho_{1}$ is a root of $F(x , 1) = 0$.

Since we assumed that $F$ has positive discriminant, the algebraic number field $\mathbb{Q}(\rho_{1})$ is real and has two fundamental units, say $\lambda_{2}$ and $\lambda_{3}$.  By Dirichlet's unit theorem, we have a sequence of mappings
$$U  \longmapsto V - \{0\} \subset \mathbb{R}^{3}$$
and
$$ \log : V - \{0\}  \longmapsto \Lambda - \{ 0 \}, $$
  where $\Lambda$ is a 2-dimensional lattice, $U  \longmapsto V - \{0\}$ is the obvious restriction of the embedding of $K$ in $\mathbb{R}^{3}$, and $\log$ is defined as follows:
   
\noindent For $(x_{1}, x_{2}, x_{3}) \in  V -\{ 0\} $,  
   $$\log(x_{1} , x_{2} , x_{3}) = (\log|x_{1}|, \log|x_{2}|, \log|x_{3}|).$$
   We define $\tau$ to be the embedding from the unit group $U$ to the lattice $\Lambda$:
  $$\tau : U \longmapsto \Lambda -\{ 0 \} .$$
   By identities (\ref {6u1}), (\ref{6u2}) and (\ref{6u3}), the vector 
$$
 \vec{u} =(\log|u_{1}|,\log|u_{3}|,\log|u_{3}|)
$$
 can be considered as 
\begin{equation}\label{6v}
\vec{v} + (\log |x - \rho'_{1}y| ,
\log |x - \rho''_{1}y| ,\log |x - \rho_{1}y|),
\end{equation}
where $$\vec{v} = (\log|D^{-1/6}(\rho_{1} - \rho''_{1}) |, 
 \log|D^{-1/6}(\rho'_{1} - \rho_{1})| , \log|D^{-1/6}(\rho''_{1} - \rho'_{1})|.$$

 We have assumed that $F(x , y)$ is monic, so we can suppose that $(1 , 0)$ is a pair of integer solutions to $F(x , y) = 1$. Note that the vector $\vec{v}$ in (\ref{6v}) is a permutation of the vector $\vec{u}(1 , 0)$.

 If $(x , y)$ is a  solution to $|F(x , y)| = 1$, then 
 $$
\vec{u} \in \vec{v} + \Lambda = \Lambda_{1}.
$$  Note that $\textrm{Vol}(\Lambda) = \textrm {Vol}(\Lambda_{1})$, where $\textrm{Vol} (\Lambda)$  is the volume of fundamental parallelepiped of lattice $\Lambda$.  Since $\vec{u}$ belongs to a 2-dimensional lattice, we can find a 2-dimensional representation for $\vec{u}$, say $(t , s)$.
 Specifically, let  $(x , y)$ be a  solution to $F(x , y) = 1$ and define functions $t$ and $s$ of $x$ and $y$ as follows
$$ t = \frac{-\sqrt{6}}{2} \log |u_{3}| \   \  , \  \   s= \frac{\log |u_{1}| - \log |u_{2}|}{\sqrt{2}}. $$
Then we have
\begin{eqnarray}\label{6u}
\log| u_{1}| &=& s / \sqrt{2} + t/ \sqrt{6}  \\ \nonumber
\log|u_{2}| & = & -s / \sqrt{2} + t/ \sqrt{6}  \\ \nonumber
\log |u_{3}| & = &-2t/ \sqrt{6} .\end{eqnarray}
Therefore, it can be easily verified that $$\vec{u} =  (\log|u_{1}| , \log|u_{2}| , \log|u_{3}|) = s \vec{\alpha} + t \vec{\beta},$$
 where $\vec{\alpha} = \frac{1}{\sqrt{2}} (1 , -1 , 0)$ and $\vec{\beta} = \frac{1}{\sqrt{6}} (1, 1 , -2)$ are two orthonormal vectors in $\mathbb{R}^{3}$. Hence, we can write $\vec{u} =(t , s)$  and $\|\vec{u}\| = \sqrt{s^{2} + t^{2}}$ , where $\| \ \|$ is the $L_{2}$ norm.
By (\ref{6u},) we get
$$ 
\left\|\left(\log\left|\frac{u_{1}}{u_{2}}\right| , \log \left|\frac{u_{2}}{u_{3}}\right|, \log\left|\frac{u_{3}}{u_{1}}\right|\right) \right\|= \sqrt{3} s \vec{\alpha'} + \sqrt{3}t \vec{\beta'} ,$$
where 
$$\vec{\alpha' } = \frac{1}{\sqrt{3}} \left(\sqrt{2} , \frac{-1}{\sqrt{2}} , \frac{-1}{\sqrt{2}}\right)$$
 and 
$$\vec{\beta'} = \frac{1}{\sqrt{3}} \left(0 , \frac{3}{\sqrt{6}} , \frac{-3}{\sqrt{6}}\right).$$
Since $\vec{\alpha'}$ and $\vec{\beta'}$ are orthonormal vectors in $\mathbb{R}^{3}$,  we get
\begin{equation}\label{6size}
 \left\|\left(\log\left|\frac{u_{1}}{u_{2}}\right| , \log \left|\frac{u_{2}}{u_{3}}\right|, \log\left|\frac{u_{3}}{u_{1}}\right|\right)\right\| = \sqrt{3}\sqrt{s^{2} + t^{2}} = \sqrt{3} \|\vec{u}\|.
\end{equation}

\textbf{Remark.} Since $\log|u_{3}| < 0$, the function $t$ is a positive-valued function.

\begin{lemma}\label{62}
If $s\geq 0$ then $\log|u_{1}| \geq \log|u_{2}|$ and
$$ 2\sinh(s/ \sqrt{2}) = \exp(-\sqrt{6}t/2),$$
and if $s < 0$ then $\log|u_{1}| < \log| u_{2}|$ and
$$ 2\sinh(-s/ \sqrt{2}) = \exp(-\sqrt{6}t/2).$$
\end{lemma}
\begin{proof}
 From (\ref{6du}), the definition of $u_{i}$,  we have  $u_{1} + u_{2} + u_{3} = 0$. Assume that $s > 0$. From (\ref{6u}), since $s$ and $t$ are both nonnegative, we get $|u_{1}| \geq |u_{2}| $ and $|u_{1}| \geq |u_{3}|$. Therefore, 
 $$
e^{ s / \sqrt{2} + t/ \sqrt{6}} -  
e^{-s / \sqrt{2} + t/ \sqrt{6}} -  
e^{-2t/ \sqrt{6}} = |u_{1}| - |u_{2}| - |u_{3}| = 0,
$$
and 
$$e^{t/\sqrt{6}}(e^{s / \sqrt{2}} - e^{-s/\sqrt{2}}) =  e^{-2t/ \sqrt{6}}. $$
Noting that   $e^{s / \sqrt{2}}- e^{-s/\sqrt{2}} = 2\sinh(s/\sqrt{2})$, will complete the proof.   One can give a similar proof for negative $s$ .
\end{proof}
 
 Let us define  
$$g(t) : = \sqrt{2} \sinh^{-1}\left(\frac{\exp(-\sqrt{6}t/2)}{2}\right) .$$
Then $s = \pm g(t)$. 

  In the following theorem, we  summarize the properties of function $g$, which will  be used later.  
\begin{thm}\label{65.2}
Let $g(t) = \sqrt{2} \sinh^{-1}\left(\frac{\exp(-\sqrt{6}t/2)}{2}\right)$. We have:
\flushleft
\begin{itemize}
\item[(i)]
$g$ is  decreasing .
\item[(ii)]
For any $t > 0$,
$$ |s| = g(t) < e^{-\sqrt{6} t/2}/\sqrt{2} .$$
\item[(iii)]
The function $g(t)e^{a t}$ is decreasing when $a \leq \frac{\sqrt{6}}{\sqrt{5}}$.
\end{itemize}
\end{thm}
\begin{proof}
\flushleft
\begin{itemize}
\item[(i)]
Since $$
\sinh (g/ \sqrt{2}) = \exp (- \sqrt{6} t/2)/2,
$$
 we have the following implicit differentiation:
$$ 
\frac{d g}{d t} \cosh(g/\sqrt{2}) = \frac{-\sqrt{3}}{2}\exp(-\sqrt{6}t/2).
$$
Since   $\cosh (g/ \sqrt{2})$ and  $\exp (-\sqrt{6} t/2)/2)$ are both positive, 
$$ \frac{d g}{d t} < 0.$$

\item[(ii)]
Define  the function 
$$ f(x)= \sqrt{2} \sinh(x/\sqrt{2}) - x.
$$ The first derivative test shows that $f$ is an increasing function and for positive $x$, $f(x) > f(0) = 0$. So 
$$\sqrt{2} \sinh (x/\sqrt{2}) > x,$$ when $x > 0$.  Put $x = |s|$ to get
$$|s| = g(t) < \sqrt{2} \sinh (|s|/ \sqrt{2}).$$

\item[(iii)]
Set
$$A( t ) = g(t) e^{a t},
$$ then
$$A'(t) = e^{at} \left(g'(t) + a g(t)\right).$$
For $a \leq 0$, $A' < 0 $ since $g' < 0$. For positive $a$, by part (i) and (ii), we have
$$
A'(t) \leq \exp{at-\sqrt{6}t/2} \left( \frac  {-\sqrt{3}}{2\cosh (g/\sqrt{2})} + \frac{a}{\sqrt{2}}\right) .  
$$
Since  $g$ is a decreasing and positive-valued function,  $\cosh(g(t)/\sqrt{2}) $ is a decreasing function of $t$. So we have 
$$\cosh(g(t)/\sqrt{2}) < \cosh (g(0)/\sqrt{2}). $$
  An easy way to evaluate  $\cosh (g(0)/\sqrt{2}) $ is to recall that $ \sinh (g(0)/\sqrt{2}) =  \exp(0) /2 =1/2$. 
Therefore, $$
\cosh (g(0)/\sqrt{2}) = \sqrt{1+\frac{1}{4}} = \sqrt{5}/2.
$$
We conclude that $A'$ is negative if 
$$ -\sqrt{3}/\sqrt{5}+ a/\sqrt{2} \leq 0.$$ 
This means
$$a \leq \frac{\sqrt{6}}{\sqrt{5}}.$$

\end{itemize}
\end{proof}


\begin{lemma}\label{69}
Let $(x,y)$ and $(x',y')$ be two distinct solutions to equation (\ref{61.2}), related to $(\eta , \xi)$. Put  $p = p(x,y)$, $p' = p(x',y')$, $q = q(x,y)$, $q'=q(x',y')$. We have $$|pq' -p'q|\geq \sqrt{3D}.$$
\end{lemma}

\begin{proof}
By definition
\begin{eqnarray*}
|pq' - p'q| &=& \left|\frac{\eta + \xi} {\sqrt{2}} \frac{\sqrt{-1}(\eta' - \xi')}{\sqrt{2}} -\frac{\eta' + \xi'} {\sqrt{2}} \frac{\sqrt{-1}(\eta - \xi)}{\sqrt{2}} \right| \\
&=&|\eta \xi' - \eta' \xi|.
\end{eqnarray*}
Since $\xi(x , y) \eta(x , y) = H(x , y)$ is a quadratic form of discriminant $-3D$ , it follows that
$$\eta \xi' - \eta' \xi = \pm \sqrt{-3D} (x y' - x' y).$$
Since  $(x,y)$ and $(x',y')$  are distinct solutions to $F(x , y)$, $x y' - x' y $ is a nonzero integer.
\end{proof}

\begin{lemma}\label{65.4}
Let $(x,y)$ and $(x',y')$ be two distinct solutions to equation (\ref{61.2}), related to $(\xi , \eta)$. Assume that $t(x' , y') \geq t(x,y)$. Then we have
$$t(x',y') \geq 2t(x,y) + \frac{\sqrt{6}}{6}\log D - \sqrt{6} \log \left(2 +\frac{1}{\sqrt{2}}\right).$$
\end{lemma}
\begin{proof}
Put 
$$p = p(x,y), \qquad p' = p(x',y'),$$ 
$$q = q(x,y), \qquad q' = q(x',y'),$$ 
$$s = s(x,y), \qquad s'= s(x',y'),$$
$$t = t(x,y), \qquad t' = t(x',y')$$
and
$$u_{i}=u_{i}(x,y), \qquad u'_{i} = u_{i}(x',y').$$
First we show that 
$$|p| \leq \sqrt{2} D^{1/6} e^{t/ \sqrt{6}} \cosh(s/\sqrt{2}).
$$ 
By the triangle inequality we have:
\begin{eqnarray*}
& &|p|  =   \frac{1}{\sqrt{2}}|\sqrt{2}p| \leq \\ \nonumber
&  &  \frac{1}{\sqrt{2}}\left( \left|(p/\sqrt{2} )+(q/\sqrt{6}) \right| +\left|(q/\sqrt{6}) - (p/\sqrt{2})\right|\right) = \\ \nonumber
& & \frac{1}{\sqrt{2}}\frac{|(q/\sqrt{6} ) +(p/\sqrt{2})|^{1/2}} {|(q/\sqrt{6}) -(p/\sqrt{2})|^{1/2}} 
 \left(\left|(q/\sqrt{6}) +(p/\sqrt{2})\right|^{1/2} \left|(q/\sqrt{6}) -(p/\sqrt{2})\right|^{1/2}\right) \\
& + &  \frac{1}{\sqrt{2}}\frac{|(q/\sqrt{6}) -(p/\sqrt{2})|^{1/2}} {|(q/\sqrt{6}) +(p/\sqrt{2})|^{1/2}}
 \left(\left|(q/\sqrt{6}) +(p/\sqrt{2})\right|^{1/2} \left|(q/\sqrt{6}) -(p/\sqrt{2})\right|^{1/2}\right).
  \end{eqnarray*}
 By (\ref{6u}) and (\ref{6ut}), we have 
 $$\left|(q/\sqrt{6}) +(p/\sqrt{2})\right|^{1/2} \left|(q/\sqrt{6}) -(p/\sqrt{2})\right|^{1/2} = D^{1/6} u_{1}^{1/2}u_{2}^{1/2} = D^{1/6} \exp(t/\sqrt{6}).$$
 Equations (\ref{6u}) and (\ref{6ut}) also give us the following identities:
 \begin{eqnarray*}& & \left(\frac{|(q/\sqrt{6} ) +(p/\sqrt{2})|^{1/2}} {|(q/\sqrt{6}) -(p/\sqrt{2})|^{1/2}} +\frac{|(q/\sqrt{6}) -(p/\sqrt{2})|^{1/2}} {|(q/\sqrt{6}) +(p/\sqrt{2})|^{1/2}}\right)\\ \nonumber
& = & e^{\frac{(\log|u_{1}| - \log|u_{2}|)}{2}}  + e^{-\frac{(\log|u_{1}| - \log|u_{2}|)}{2}} \\ \nonumber
&=& 2\cosh(\frac{s}{\sqrt{2}})
\end{eqnarray*}
 and
$$|q| = (\sqrt{6}/2)D^{1/6}e^{-2t/\sqrt{6}}.$$
Using Lemma \ref{69}, we get
$$D^{1/6} \leq  e^{(t' - 2t)/\sqrt{6}} \cosh\left(|s'|/\sqrt{2}\right) +  e^{(t - 2t')/\sqrt{6}} \cosh\left(|s|/\sqrt{2}\right).$$
One can express the above equation in terms of $\sinh$ instead of $\cosh$ by substituting  $\cosh(|s|/\sqrt{2})$  with 
$$ \sinh\left(\frac{|s|}{\sqrt{2}}\right) + e^{-|s|/\sqrt{2}}.
$$
  Now we use the assumption that $t' \geq t$ and the fact that  $e^{-|s|/\sqrt{2}} \leq 1$. By Lemma \ref{62}, we get
$$D^{1/6} \leq e^{(t' - 2t)/\sqrt{6}}\left(1 + e^{-3(t' -t)/\sqrt{6}}\right) \left(1 + \frac{e^{-\sqrt{6}t/2}}{2}\right).$$
Note that by Theorem (\ref{65.2}), $t \geq \log(2)/\sqrt{6}$, whereby taking the logarithm of both sides of the above equality, yields
$$t' -2t \geq \sqrt{6}/6 \log(D) - \sqrt{6}\log \left(\left(1 + e^{-3(t'-t)/\sqrt{6}}\right)  \left(1 + \frac{1}{2\sqrt{2}}\right)\right).$$
Therefore,
 $$t' - 2t \geq \sqrt{6}/6 \log(D) - \sqrt{6} \log \left(2 + \frac{1}{\sqrt{2}}\right  ).$$
\end{proof}

 \begin{lemma}\label{6tr}Suppose that (\ref{61.2}) has three distinct solutions related to $(\xi , \eta)$. Then  three distinct corresponding points $(t , s)$ , $(t' , s')$ and $(t'' , s'')$  form a triangle. 
\end{lemma}
\begin{proof} 
 Suppose $(t , s)$ , $(t' , s')$ and $(t'' , s'')$  are collinear and $t \leq t' \leq t''$ . Then
  $$\frac{s' - s}{t' - t}  = \frac{s'' - s'}{t'' - t'} .$$
 Assume, without loss of generality, $s' > 0$. Since $g(t) = |s|$ is a decreasing function of $t$, we have
 $$s'' -s' < 0 $$
 and consequently, $s' -s < 0$. Therefore, 
 $$s > 0 .$$
 Since we assumed $(t , s)$ , $(t' , s')$ and $(t'' , s'')$  to be collinear,
 $$\frac{s'' - s}{t'' - t}  = \frac{s' - s}{t' - t} $$ 
 By Lemma \ref{65.4}, and since $|s| \geq |s'| \geq |s''| > 0$, we get
 $$\frac{s' - s}{t' - t}  < \frac{-s'}{t} < \frac{-2s'}{t'} <  \frac{s'' - s'}{t'' - t'} < 0 .$$ 
 This contradiction shows that $(t , s)$ , $(t' , s')$ and $(t'' , s'')$  are not  collinear (note that any vertical or horizontal line intersects the graph of $g$ and $g'$ at most in two points).
\end{proof}

  Suppose that (\ref{61.2}) has three distinct solutions related to $(\xi , \eta)$ and $A$ is the area of the triangle formed by three distinct corresponding points $(t , s)$, $(t' , s')$ and $(t'' , s'')$. Then vectors  $(t -t' , s - s')$ and $(t -t'' , s - s'')$  generate a sub-lattice of $\Lambda_{1}$ with the volume of fundamental parallelepiped equal to $2A$. Therefore, $$2 A \geq \textrm{Vol} (\Lambda_{1}).$$
Now let us estimate $2 A$, the area of rectangle which has   $(t , s)$ , $(t' , s')$ and $(t'' , s'')$ as three of its edges. Recall that $s(x , y) = \pm g(t(x , y))$ and $g$ is a decreasing function. Suppose that $t \leq t' \leq t''$. Then $g(t'') \leq g(t') \leq g(t)$ and we have
 $$2 A \leq (t'' - t)(g(t) + g(t')) = (t'' - t) (|s| + |s'|).$$ 
 Part (iii) of Theorem \ref{65.2} shows that 
$$|s'| < |s| e^{\frac{-\sqrt{6}(t' - t)}{\sqrt{5}}},
$$
 Therefore,
 $$\textrm{Vol}(\Lambda) = \textrm{Vol}(\Lambda_{1}) \leq (t'' - t) |s|  \left(1 + e^{\frac{-\sqrt{6}
 (t'-t)}{\sqrt{5}}}\right).$$ 
  Using Theorem  \ref{65.2} again,   we get the following gap principle of this paper which is essentially Theorem 5.5  of \cite{6Ok}: 
 \begin{thm}\label{65.5}
 Suppose that $F(x , y)$ has three distinct solutions $(x , y)$, $(x' , y')$ and $(x'' , y'')$, all related to $(\xi , \eta)$.  Assume that 
$$ t = t(x, y) \leq t' = t(x' , y') \leq  t'' = t(x'' , y''),$$
  where $t$ is the function defined in  the begining of this section. We have
 $$t'' \geq \frac{\sqrt{2} \textrm{Vol}(\Lambda) exp(\sqrt{6}t/2)}{1 + \rm{exp}(-\sqrt{6}(t' - t)/\sqrt{5})},$$
 where $\textrm{Vol}(\Lambda)$ is the volume of fundamental parallelepiped of lattice $\Lambda$. 
 \end{thm}


\section{Linear Forms In Logarithms}\label{6s5}

We have seen that $ \sqrt{2} s = \log|u_{1}| - \log |u_{2}| =\log \lambda_{1} + m\log\lambda_{2} + n \log \lambda_{3}$. Where $s$ is a function of $(x , y)$ defined in Section 3 and $u_{i}$ are also functions of $(x , y)$ defined in Section 
\ref{6s2}.
By Lemma \ref{65.2}, we have 
$$\log (\sqrt{2}|s| )\leq -(\sqrt{6}/2)t.$$
Here, we will use a well-known lower bound for linear forms in logarithms of algebraic numbers, to find an upper bound for $\log (\sqrt{2} |s|)$.
\begin{thm}[Matveev]\label{6mat}
Suppose that $\mathbb{K}$ is a real algebraic number field of degree $d$. We are given numbers $\alpha_{1} , \ldots \alpha_{n} \in \mathbb{K}^{*}$ with absolute logarithmic heights $h(\alpha_{j})$.

\noindent Let $\log \alpha_{1}$, $\ldots$ , $\log \alpha_{n}$ be arbitrary fixed non-zero values of the logarithms. Suppose that 
$$A_{j} \geq \max \{dh(\alpha_{j}) , |\log \alpha_{j}| \}, \  \   1 \leq j \leq n.$$
Now consider the linear form 
$$L = b_{1}\log\alpha_{1} + \ldots + b_{n}\log\alpha_{n},$$
with $b_{1}, \ldots , b_{n} \in \mathbb{Z}$ and with the parameter $B = max\{1 , max\{b_{j}A_{j}/A_{n}: \  1\leq j  \leq n\}\}$ . Put
$$\Omega = A_{1} \ldots A_{n},$$
$$C(n) = \frac{16}{n!}e^{n}(2n + 2)(n + 2) (4n + 4)^{n + 1}(\frac{1}{2}en),$$
$$C_{0} = \log (e^{4.4n + 7}n^{5.5}d^{2}\log (en)),$$
$$W_{0} = \log(1.5eBd\log(ed)).$$
If $b_{n} \neq 0$, then 
$$\log|L| > -C(n) C_{0} W_{0}d^{2}\Omega.$$
\end{thm}

Here, we recall the definition of absolute logarithmic height from \cite{6Mat1, 6Mat2}. Let $\mathbb{Q}(\rho)^{\sigma}$ be the embeddings of the real number field $\mathbb{Q}(\rho)$ in $\mathbb{R}$,  $1 \leq \sigma \leq 3$, where $\rho$ is a root of $F(x , 1) = 0$. We respectively have $3$ Archimedean valuations of $\mathbb{Q}(\rho)$:
$$|\alpha|_{\sigma} = |\alpha^{(\sigma)}| , \   \   1 \leq \sigma \leq 3 .$$
We enumerate simple ideals of $\mathbb{Q}(\rho)$ by indices $\sigma > 3$ and define
 non-Archimedean valuation of $\mathbb{Q}(\rho)$ by the formulas
 $$ |\alpha | _{\sigma} = (  \textrm{Norm} \  \mathfrak{p})^{-k}, $$ 
 where
 $$\  k = \textrm{ord}_ {\mathfrak{p}} (\alpha) , \   \mathfrak{p} = \mathfrak{p}_{\sigma} , \   \sigma > d ,$$
 for any $\alpha \in \mathbb{Q}^{*} (\rho)$.  Then we have the \emph{product formula} :
 $$ \prod_{1}^{\infty} |\alpha|_{\sigma} = 1 ,  \   \alpha \in  \mathbb{Q}(\rho) .$$
 Note that $|\alpha|_{\sigma} \neq 1$ for only finitely many $\alpha$ .
We define the \emph{absolute logarithmic height} of $\alpha$ as
 $$h(\alpha) = \frac{1}{6} \sum _{\sigma = 1}^{\infty} \left|\log |\alpha|_{\sigma}\right|. $$
 We will apply  Matveev's lower bound to 
 $$\log|u_{1}| - \log|u_{2}| = \log \lambda_{1} + m_{1}\log \lambda_{2} + n_{1}\log \lambda_{3} . $$
Suppose that
$$\|\vec{u}(x_{0} , y_{0})\| = \min _{(x ,y)\in S} \|\vec{u}(x , y)\|$$
 and 
 $$\log|u_{1}(x_{0} , y_{0})| - \log |u_{2}(x_{0} , y_{0})| = \log \left|\frac{\rho - \rho''}{\rho' - \rho}\right| + a \log \lambda_{1} + b \log \lambda_{2}$$
 then for any solution $(x , y)$, we can write
$$\log|u_{1}(x , y)|  - \log |u_{2}(x , y)| =  \log \lambda + m \log \lambda_{1} + n\log \lambda_{2},$$
where $m = m_{1}- a$, $n = n_{1}-a$ and
$$\lambda =   \left|\frac{\rho - \rho''}{\rho' - \rho}\right|\lambda_{1}^{a} \lambda_{2}^{b}.$$

 Since $\lambda_{2}$ and $\lambda_{3}$ are the fundamental units of the ring of integers of $\mathbb{Q}(\rho)$, $\lambda_{1}$, $\lambda_{2}$ and $\lambda_{3}$ are multiplicatively dependent if and only if $\lambda_{1}$ is a unit.  If $\lambda_{1}$ is a unit then we can write $\log|u_{1}| - \log|u_{2}|$ as a linear form in two logarithms. Since Theorem \ref{6mat} gives a better lower bound for linear forms in two logarithms, we can assume that $\lambda_{1}$, $\lambda_{2}$ and $\lambda_{3}$ are multiplicatively independent and $\log|u_{1}| - \log|u_{2}|$ is a linear form in three logarithms.

First, suppose that $\lambda$ is a unit in the number field. We have
$$h(\lambda) = \frac{1}{6} \left( \left|\log |\lambda|\right| + \left|\log  |\lambda '|\right| + \left|\log |\lambda ''|\right|\right) = \frac{1}{6}\left|\tau(\lambda)\right|_{1} , $$
where $\lambda '$ and $\lambda ''$ are the conjugates of $\lambda$ , $\tau$ is the embedding of units to the  lattice $\Lambda$ and $|  \  |_{1}$ is the $L_{1}$ norm on $\mathbb{R}^{3}$ .  So we have
$$h(\lambda) = \frac{1}{6}\left|\tau(\lambda)\right|_{1} \leq \frac{\sqrt{3}}{6} \left\|\tau(\lambda)\right\|,$$
where $\| \  \|$ is the $L_{2}$ norm on $\mathbb{R}^{3}$. 
So when $\lambda$ is a unit
\begin{equation}\label{h1}
\max \{3h(\lambda) , |\log(|\lambda|)| \} \leq \|\tau(\lambda)\| ,\end{equation}
since $\left|\log(\lambda)\right| \leq \sqrt{\log^{2}|\lambda| + \log^{2}|\lambda '| + \log^{2}|\lambda''|} = \left\|\tau(\lambda)\right\|$.

In the identity 
$$\log|u_{1}| - \log |u_{2}| =\log \lambda +  m_{1} \log \lambda_{2} + n_{1} \log \lambda_{3}, $$
 $\lambda_{2}$ and $\lambda_{3}$ are fundamental units of $\mathbb{Q}(\rho)$. Therefore,  in  Theorem \ref{6mat}, $A_{i}$ can be taken equal to $|\tau(\lambda_{i})|_{2}$, for $i = 2 , 3$.

Bases $\vec{b_{1}} $ and $\vec{b_{2}}$ of lattice $\Lambda$ are called reduced if the following conditions are satisfied :
\flushleft
\begin{itemize}
\item [(i)] $\left\|\vec{b_{1}}\right\| \leq  \left\|\vec{v}\right\|$ for every vector $\vec{v} \in \Lambda  -  \{\vec{0}\}$;
\item[(ii)]
 $\left\|\vec{b_{2}}\right\| \leq \left\|\vec{v}\right\|$
  for every vector $\vec{v} \in \Lambda  - \mathbb{Z}\vec{b_{1}}$.
\end{itemize}

\textbf{Remark.}
 Although the definitions of reduced basis for lattices and reduced forms are somehow related, one should note that we define them separately and they are not to be confused.
 
It is a fact that we can always choose  a pair of reduced basis for a two dimensional lattice. So we choose the fundamental  units $\lambda_{2}$ and $\lambda_{3}$ such that the basis $\tau(\lambda_{2})$ and $\tau(\lambda_{3})$ are reduced basis for $\Lambda$.
When $\vec{b_{1}}$ and $\vec{b_{2}}$ are the reduced basis of $\Lambda$, since $\vec{b_{1}} , \vec{b_{2}} \leq \vec{b_{1}} \pm \vec{b_{2}}$, we conclude that the angle between vectors $\vec{b_{1}}$ and $\vec{b_{2}}$ must be between $\pi/3$ and $2\pi/3$. Therefore, 
$\lambda_{2}$ and $\lambda_{3}$ can be chosen so that 
$$
\left\|\tau(\lambda_{2})\right\| \left\|\tau(\lambda_{3})\right\| \leq  \frac{2}{\sqrt{3}} \textrm{Vol}(\Lambda).$$
Hence, in our case, 
$$A_{2} A_{3} \leq \frac{2}{\sqrt{3}} \textrm{Vol}(\Lambda).$$


By (\ref{6size}),  we have 
$$\left\|(\log\frac{|u_{1}|}{|u_{2}|} ,\log\frac{|u_{2}|}{|u_{3}|} , \log\frac{|u_{3}|}{|u_{1}|} )\right\| = \sqrt{3} \|\vec{u} \|.$$
The well-known inequality  $\frac{a + b+ c}{3} \leq [\frac{a^{2} + b^{2} + c^{2}}{3}]^{1/2}$ shows that
\begin{equation}\label{6L}
|\vec{v}|_{1} \leq \sqrt{3} \|\vec{v}\|
\end{equation}
for every vector $\vec{v} \in \mathbb{R}^{3}$.
Therefore,
\begin{eqnarray*}
& &\left|\log|(\rho - \rho'')(x - \rho' y)| \right|+ \left|\log |(\rho' - \rho)  (x - \rho'' y)| \right| + \left| \log |(\rho' - \rho'')  (x - \rho y)|\right| \\
&\leq& 3 \left\|\vec{u}(x_{0} , y_{0})\right\|. \end{eqnarray*}
Now, we note that 
\begin{eqnarray*}
&&\sum_{\sigma > 3} {\left|\log\left|\frac{(\rho - \rho'')  (x - \rho' y)}{(\rho' - \rho)  (x - \rho'' y)}\right|_{\sigma}\right|} \\
& \leq& \sum_{\sigma >3}{\left|\log\left|(\rho - \rho'')   (x - \rho' y)\right|_{\sigma}\right|} + \sum_{\sigma >3}{\left|\log\left|(\rho' - \rho ) (x - \rho'' y)\right|_{\sigma}\right|} .\end{eqnarray*}
We know that  Archimedean  valuations of $\left|(\rho - \rho '') (x - \rho' y) \right|$ are itself,  $\left|(\rho' - \rho'')    (x - \rho y)\right|$  and 
$\left|\rho' - \rho  (x - \rho' y)\right| $.  So by the product formula, since $(x , y)$ is a solution to (\ref{61.2}), the product of all non-Archimedean valuations equals
$D^{-1/2}$.
Therefore,
$$\left|\sum_{\sigma > 3} \log\left|(\rho - \rho '')  (x - \rho' y)\right|_{\sigma} \right| = \frac{1}{2} \log D, $$ 
and similarly
$$\left|\sum_{\sigma > 3} \log\left|(\rho' - \rho)  (x - \rho'' y) \right|_{\sigma} \right| = \frac{1}{2} \log D. $$ 
Since    
$(\rho' - \rho)  (x - \rho'' y)$ and  $(\rho - \rho '')  (x - \rho' y)$  are algebraic integers,  we get
\begin{equation}\label{617}
h(\lambda_{1}) \leq \frac{1}{6} \left(3\|\vec{u}(x_{0} , y_{0})\| +  \log{D} \right).
\end{equation}
This gives an estimate for $A_{1}$.

 Let  $B_{1} = B A_{3}$, where $B$ is as in theorem \ref{6mat}. Then
$$B_{1} = \max \{b_{j} A_{j}, \  : 1 \leq j \leq 3 \} .$$
Since 
$$\log\left|\frac{u_{1}}{u_{2}}\right| = \log \left|\frac{(\rho -\rho'')  (x - \rho' y)}{(\rho' - \rho)   (x - \rho'' y)}\right| + m \lambda_{1} + n \lambda_{2},$$
 we can write
\begin{eqnarray*}
& &\left(\log\left|\frac{u_{1}}{u_{2}}\right| , \log \left|\frac{u_{2}}{u_{3}}\right| , \log \left|\frac{u_{3}}{u_{1}}\right|\right) = \\ \nonumber
 & &  \left(\log\frac{|(\rho - \rho'')   (x - \rho' y)|}{|(\rho' - \rho)   (x - \rho'' y)|} ,\log\frac{|(\rho' - \rho)  (x - \rho'' y)|}{|(\rho' - \rho'')  (x - \rho y)|} , \log\frac{|(\rho' - \rho'')  (x - \rho y)|}{|(\rho - \rho'')  (x - \rho' y)|} \right)  \\ \nonumber
& + &  m \vec{\lambda_{1}}  + n \vec{\lambda_{2}} ,
\end{eqnarray*}
where  $\vec{\lambda_{i}} = \tau(\lambda_{i})$ , for $ i = 2 , 3$. 
Since $\lambda_{2}$ and $\lambda_{3}$ have been chosen so that $\vec{\lambda_{2}}$ and $\vec{\lambda_{3}}$ form a reduced basis for the lattice $\Lambda$, we get
 \begin{eqnarray*}
 & & m |\vec{\lambda_{2}}|_{1} \,  , n |\vec{\lambda_{3}}|_{1}   \leq  \\ \nonumber
& & \left|\left(\log\frac{|(\rho - \rho'')  (x - \rho' y)|}{|(\rho' - \rho)  (x - \rho'' y)|} ,\log\frac{|(\rho' - \rho)  (x - \rho'' y)|}{|(\rho' - \rho'')  (x - \rho y)|} , \log\frac{|(\rho' - \rho'') (x - \rho y)|}{|(\rho - \rho'')  (x - \rho y)|} \right)\right|_{1}  \\ \nonumber
&+ & \left|\left(\log\left|\frac{u_{1}}{u_{2}}\right| , \log \left|\frac{u_{2}}{u_{3}}\right| , \log \left|\frac{u_{3}}{u_{1}}\right|\right) \right|_{1}.\end{eqnarray*}
Therefore, by (\ref{6size}) and (\ref{6L})
\begin{equation}\label{6mn}
m |\vec{\lambda_{2}}|_{1}  , n |\vec{\lambda_{3}}|_{1} \leq 3 \left(\|\vec{u}\| +\| \vec{u}(x_{0} , y_{0})\|\right).
\end{equation}

\begin{thm}\label{6dis}
Let $F$ be a cubic binary equation with positive discriminant.
For all pairs of solution  $(a , b)$ to the equation (\ref{61.2}), except possibly one of them,  we have 
$$ D \leq 64 e^{2\sqrt{6}t} ,$$
where $D$ is the discriminant of $F(x , y)$ and $t = t(a , b)$, for the function $t$ defined in Section \ref{6s3}. Moreover, when $t \geq 5$
$$ D \leq \frac{1}{2} e^{2\sqrt{6}t} .$$
\end{thm}
\begin{proof}
By (\ref{63}) 
$$|H| = |\xi \eta| = |\frac{p + i q}{\sqrt{2}}  .  \frac{p -iq}{\sqrt{2}}| = \frac{p^{2} + q^{2}}{2} . $$
By (\ref{6du}),
$$q = \frac{\sqrt{6}}{2}D^{1/6} u_{3}$$
and 
$$|p| = |\frac{\sqrt{2}}{2} (u_{1} - u_{2})| D^{1/6} \leq \frac{\sqrt{2}}{2}D^{1/6}(|u_{1}| + |u_{2}|).$$ 
Therefore, by (\ref{6u})  
\begin{equation} \label{6ht}
|H| \leq \frac{1}{2}D^{1/3}\left(e^{2t/\sqrt{6}}(e^{2s/\sqrt{2}} + e^{-2s/\sqrt{2}} +2)/2 + \frac{3}{2} e^{-4t/6}\right).
\end{equation}
Therefore, by Lemma \ref{6mike}, for all solutions $(a , b)$ to (\ref{61.2}), except at most one of them
\begin{equation}\label{6bd}
\frac{1}{2} D^{1/2}\sqrt{3} \leq \frac{1}{2}D^{1/3}\left(e^{2t/\sqrt{6}}(e^{2s/\sqrt{2}} + e^{-2s/\sqrt{2}} +2)/2 + \frac{3}{2} e^{-4t/6}\right).
\end{equation}
Since $t > 0$, part (ii) of Lemma \ref{61.2} says that $|s| < e^{-\sqrt{6} t/2}/\sqrt{2}$. Hence,
$$D^{1/6} \leq 2 e^{2t/\sqrt{6}},$$
which proves the theorem for general $t$. When $t \geq 5$, we note that  $\frac{3}{2} e^{-4t/6} < 0.054$ and $|s| <0.0016$ by Theorem \ref{65.2} .
\end{proof}

Since $\vec {u} =(t , s)$  and $\|\vec{u}\| = \sqrt{t ^{2} + s^{2}}$, from Theorem \ref{65.2}, we deduce that $|\vec{u}|_{2}$ is an increasing function of $t$. So we can assume that Theorem \ref{6dis} is satisfied for all solutions, except  possibly $(x_{0} , y_{0})$, where
$$\|\vec{u}(x_{0} , y_{0})\| = \min _{(x ,y)\in S} \|\vec{u}(x , y)\|,$$
$S$ is the set of all solutions to (\ref{61.2}) and $\vec{u} = (\log|u_{1}| , \log|u_{2}| , \log|u_{3}|)$. 

Suppose that three distinct solutions $(x , y)$, $(x' , y')$ and $(x'' , y'')$ of (\ref{61.2})  are related to $(\xi , \eta)$ and $t'' = t( x'' , y'') > t' = t(x' , y') > t = t(x , y)$. 
First, we recall that $\vec {u} =(t , s)$  and $|\vec{u}|_{2} = \sqrt{t ^{2} + s^{2}}$. By Theorem \ref{65.2}, if we take $t \geq 5$ (Theorem \ref{6dis} enables us to assume that $t$ is large), we get
$$|\vec{u}|_{2} = \sqrt{t^{2} +e^{-\sqrt{6}t}/2} \leq 1.0016 t.$$
Therefore,  by (\ref{617}) and theorem \ref{6dis}, we can take 
$$A_{1} = 1.0016 (\frac{3}{2} t + \sqrt{6}t) . $$

Inequality  (\ref{6mn}) suggests  the value $3 (1.0016) (t'' + t)$ for $B_{1}$. But by theorem \ref{65.4}, for large discriminant  $D$, $t'' > 4 t$ . So we  take 
\begin{equation}\label{B1}
B_{1} =   3(1.0016)(t'' + t''/4).
\end{equation}


So Matveev 's lower bound gives us:
$$\log |L | >  -1.5036 \times 10^{11}A_{1} A_{2} A_{3}  \log \left(25.6708 (3.0048)(1.25)t''/ A_{3}\right),$$
where  $ L = \log|u_{1}| - \log |u_{2}| $.
On the other hand, by Theorem  \ref{65.2}, we have 
$$\log |L| = \log \sqrt{2}|s''| \leq -\sqrt{6}t''/2 .$$
 We conclude that
$$ t'' \leq 1.2276 \times  10^{11}  A_{1} A_{2} A_{3} \log (96.2751 t''/A_{3}) ,$$
or 
$$\frac{96.2751 t''/A_{3}}{\log (96.2751 t''/A_{3})} \leq  1.1892 \times  10^{13}  A_{1} A_{2}.$$
Therefore,
$$\log (96.2751 t''/A_{3}) \leq  \frac{e}{e - 1} \log(1.1892 \times  10^{13}  A_{1} A_{2}). $$
Recalling that $A_{2}A_{3} \leq 2/\sqrt{3} \textrm{Vol} (\Lambda)$, we obtain the following upper bound for $t''$:
\begin{equation}\label{6ut}
 t'' \leq \frac{e}{e - 1}1.2276 \times  10^{11}  \left(\frac{2}{\sqrt{3}}\right) \textrm{Vol}(\Lambda)  A_{1}\log(1.1892 \times  10^{13}  A_{1} A_{2}). 
 \end{equation}


\section{Proof Of The Main Results }

Let
$$\|\vec{u}(x_{0} , y_{0})\| = min _{(x ,y)\in S} \|\vec{u}(x , y)\|.$$
Suppose that $(x , y) $, $(x' , y')$ and $(x'' , y'')$, with none of them equal to $(x_{0} , y_{0})$, are three distinct solutions to (\ref{61.2}), and related to a fixed choice of resolvent form. Let $t = t (x , y) < t' = t( x' , y') < t'' = t(x'' , y'') $. By (\ref{6ut}) and Theorem \ref{65.5},  we get 
\begin{eqnarray}\label{6mi}\nonumber
&&\frac{e}{e - 1}1.2276 \times  10^{11}  A_{1}(\frac{2}{\sqrt{3}})\textrm{Vol}(\Lambda)  \log(1.1892 \times  10^{13}  A_{1} A_{2}) \\ 
  & &\geq \frac{\sqrt{2} \textrm{Vol}(\Lambda) exp(\sqrt{6}t/2)}{1 + \exp(-\sqrt{6}(t' - t)/\sqrt{5})},
  \end{eqnarray}
  where $A_{1} = (3/2 + \sqrt{6})(1.006) t $ and $ A_{2} = \left\|\tau (\lambda_{2})\right\|$. Without loss of generality, we can assume that  $\left\|\tau (\lambda_{2})\right\| \leq  \left\|\tau (\lambda_{3})\right\|$.  Therefore,
 $$\frac {\sqrt{3}}{2} |\tau (\lambda_{2})|_{2} ^{2} \leq \frac{\sqrt{3}}{2} \left\|\tau (\lambda_{2})\right\|\left\|\tau (\lambda_{3})\right\| \leq  \textrm{Vol}(\Lambda).$$
 We have
 $$\log|u_{1}(x , y)| - \log|u_{2}(x , y)| = \log \lambda_{1} + m' \log \lambda_{2} + n' \log \lambda_{3} ,$$
 where $m'$ and $n'$ are integers.
 Since $(x , y) \neq (x_{0} , y_{0})$,  at least one of $m'$ or $n'$ is a nonzero integer. So by (\ref{6mn}),  we have 
 $ \left\|\tau (\lambda_{2})\right\| \leq 6.01  t$ . Using Theorem \ref{65.4}, we get 
 
 \begin{eqnarray*}
& &\frac{e}{1-e} 4.8484 \times  10^{11}(\frac{2}{\sqrt{3}})\textrm{Vol}(\Lambda) t \log(2.8184 \times  10^{14}  t^{2})\\
 & & \geq \frac{\sqrt{2} \textrm{Vol}(\Lambda) exp(\sqrt{6}t/2)}{1 + \exp(-\sqrt{6}t/\sqrt{5})}.\end{eqnarray*}
  Therefore, $t < 27.91$ and by equation (\ref{6bd}), $D <  5.31\times 10^{59}$;
   i.e. we have proven that there are at most 2 pairs of solutions $(x , y ) \neq (1 , 0)$ and $(x' , y') \neq (1 , 0)$ related to a resolvent form $(\xi , \eta)$ , when $D \geq 5.31 \times 10^{59}$.

If we suppose that  $D <  5. 31\times 10^{59}$, by (\ref{617}), we can take 
   $$A_{1} = \frac{3}{2} t +\frac{1}{2} \log (5.31 \times 10^{59}) .$$
By substituting this new value of $A_{1}$ in (\ref{6ut}), we get 
$$ t \leq 27.5321 , $$
and therefore, $D < 1.4\times10^{57}$.   Since we have three pairs of resolvent forms, Theorem (\ref{6main}) is proved.

  As we mentioned in the remark after the  proof of Lemma \ref{6mike}, the solution $(1 , 0)$ needs to be treated separately, only if $F$ is equivalent to a monic reduced form. Otherwise, $(x_{0} , y_{0}) \neq (1 , 0)$ and Lemma \ref{6mike} and therefore Lemma \ref{6dis} will hold for all solutions without any exception.  By the analytic class number formula and Louboutin's upper bound (which can be found in \cite{6Co}) :
  $$\textrm{Vol}(\Lambda) \leq  \frac{\sqrt{3}}{8}\sqrt{D} \log ^{2}D,$$
 we have that
 $$A_{2}  \leq \frac{1}{4} D^{1/4} \log D. $$
By Theorem \ref{6dis},$$A_{2}\leq \frac{1}{4}    e^{\sqrt{6}t /2}(\log\frac{1}{2} + 2\sqrt{6}t).$$
Now, having appropriate  values of $A_{1}$ and $A_{2}$ in hand, we solve inequality (\ref{6mi}) to get $t \leq 28.38$ and consequently by (\ref{6bd}), $D \leq 9\times 10^{58}$; i.e. we have proven that there are at most 2 pairs of solutions $(x , y ) $ and $(x' , y') $ related to a resolvent form $(\xi , \eta)$ , when $D > 9 \times 10^{58}$. Therefore, we get Theorem \ref{66}.

  In \cite{6Ben}, it is  proved that if $D \geq 2400$, related to a fixed pair of resolvent form, there are at most 3 different pairs of solutions $(x , y)$ to (\ref{61.2}) with $H(x , y) \geq \frac{1}{2} \sqrt{3D}$, where $H$ is the Hessian of $F$.  This together with lemma \ref{6mike} leads to the main theorem of \cite{6Ben}, that is, the equation $F(x , y) = 1$ has at most 10 solutions in integer $x$ and $y$ .

 For $0 <D < 2400$, equation $F(x , y) = 1$ is completely solved for representatives of every equivalent class of binary cubic forms. These computations show that the equation (\ref{61.2}) with discriminant $0< D< 10^{6}$ has at most $9$ solutions in integers $x$ and $y$. The complete result of these computations are tabulated in section 9 of \cite{6Ben}.


\begin{thebibliography}{99}
\bibitem{6Del} B.N.~Delone and D.K. Fadeev. The Theory of Irrationalities of the Third Degree. Translation of math. Monographs, AMS $10$ (1964).

\bibitem{6Ben} M.A.~Bennett. On the representation of unity By binary Cubic Forms. Trans. Amer. Math. Soc. $353$ (2001), 1507-1534.

\bibitem{6Co} J.H.E.~Cohn. The Diophantine equation $x^{2} + C = y^{n}$, II, Acta Arith. $109.2$ (2003), 205-206.

\bibitem{6Ba} V.I.~Baulin. On an intermediate equation of the third degree with least positive discriminant (Russian). Tul'sk Gos.Ped.Inst. Ucen. Zap. Fiz. Math. Nauk. Vip. $7$ (1960), 138-170.

\bibitem{6Ev} J.H.~Evertse. On the representation of integers by binary cubic forms of positive discriminant. Invent. Math.$73$(1983), 117-138.

\bibitem{6Ga}I.~Ga\'{a}l and N. Schulte. Computing all power integral of cubic fields. Math. Comp. $53$ (1989), 689-696.

\bibitem{6Lee} E.~Lee. Studies on Diophantine Equations. PhD thesis, Cambridge University, 1992. 

 \bibitem{6Li} F.~Lippok.
On the representation of $1$ by binary cubic forms of positive discriminant.
J. Symbolic Computation $15$ (1993), 297-313. 

\bibitem{6Mat1} E.M.~Matveev, An explicit lower bound for a homogeneous rational linear form in logarithms of algebraic numbers, Izv. Ross. Akad. Nauk Ser. Mat. $62$ (1998), 81-136, translation in Izv. Math. $62$ (1998), 723-772.

\bibitem{6Mat2} E.M.~Matveev, An explicit lower bound for a homogeneous rational linear form in logarithms of algebraic numbers, Izv. Ross. Akad. Nauk Ser. Mat. $64$ (2000), 125-180, translation in Izv. Math. $64$ (2000), 1217-1269.


\bibitem{6Mig} M.~Mignotte. Verification of a conjecture of E. Thomas . J. Number Theory $44$ (1993), 172-177.
    
\bibitem{6Mig1} M.~Mignotte. Peth\~{o}'s cubics. Dedicated to Professor Kalman Gy\H{o}ry on the occasion of his 60th birthday. Publ. Math. Debrecen $56$ (2000), 481-505.

\bibitem{6Tz} M.~Mignotte and N. Tzanakis. On a family of cubics. J. Number Theory $39$(1991), 41-49.


\bibitem{6Lj} W.~Ljunggren. Einige Bemerkungen uber die Darstellung ganzer zahlen durch bibare kubische formen mit positiver Diskriminante. Acta Math. $75$(1942), 1-21.

\bibitem{6Ok} R.~Okazaki. Geometry of a cubic Thue  equation, Publ. Math. Debrecen. $61$ (2002),267-314.
 

\bibitem{6Si} C.L.~Siegel. Uber einige Anwendungen diophanticher Approximationen. Abh.Preuss.Akad.Wiss. (1929), Nr.1.

\bibitem{6Tho} E.~Thomas, Complete solutions to a family of cubic Diophantine equations, J. Number Theory $34$ (1990), 235-250. 
\end{thebibliography}
\end{document}